\documentclass[oneside,reqno,american]{amsart}
\usepackage[T1]{fontenc}
\usepackage[utf8]{inputenc}
\usepackage{xcolor}
\usepackage{babel}
\usepackage{prettyref}
\usepackage{mathtools}
\usepackage{amstext}
\usepackage{amsthm}
\usepackage{amssymb}
\usepackage[pdfusetitle,
 bookmarks=true,bookmarksnumbered=false,bookmarksopen=false,
 breaklinks=false,pdfborder={0 0 0},pdfborderstyle={},backref=false,colorlinks=false]
 {hyperref}
\hypersetup{
 colorlinks=true,citecolor=blue,linkcolor=blue,linktocpage=true}

\makeatletter
\numberwithin{equation}{section}
\numberwithin{figure}{section}

\newrefformat{cor}{Corollary~\ref{#1}}
\newrefformat{subsec}{Section~\ref{#1}}
\newrefformat{lem}{Lemma~\ref{#1}}
\newrefformat{thm}{Theorem~\ref{#1}}
\newrefformat{sec}{Section~\ref{#1}}
\newrefformat{chap}{Chapter~\ref{#1}}
\newrefformat{prop}{Proposition~\ref{#1}}
\newrefformat{exa}{Example~\ref{#1}}
\newrefformat{tab}{Table~\ref{#1}}
\newrefformat{rem}{Remark~\ref{#1}}
\newrefformat{def}{Definition~\ref{#1}}
\newrefformat{fig}{Figure~\ref{#1}}
\newrefformat{claim}{Claim~\ref{#1}}

\makeatother

\theoremstyle{plain}
\newtheorem{thm}{\protect\theoremname}[section]
\theoremstyle{definition}
\newtheorem{defn}[thm]{\protect\definitionname}
\theoremstyle{remark}
\newtheorem{rem}[thm]{\protect\remarkname}
\theoremstyle{plain}
\newtheorem{lem}[thm]{\protect\lemmaname}
\newtheorem{cor}[thm]{\protect\corollaryname}
\theoremstyle{definition}
\newtheorem{example}[thm]{\protect\examplename}
\newrefformat{cor}{Corollary \ref{#1}}
\newrefformat{def}{Definition \ref{#1}}
\newrefformat{exa}{Example \ref{#1}}
\newrefformat{lem}{Lemma \ref{#1}}
\newrefformat{rem}{Remark \ref{#1}}
\newrefformat{sec}{Section~\ref{#1}}
\newrefformat{subsec}{Section~\ref{#1}}
\newrefformat{thm}{Theorem \ref{#1}}
\providecommand{\corollaryname}{Corollary}
\providecommand{\definitionname}{Definition}
\providecommand{\examplename}{Example}
\providecommand{\lemmaname}{Lemma}
\providecommand{\remarkname}{Remark}
\providecommand{\theoremname}{Theorem}

\begin{document}
\title{Random Nonlinear Fusion Frames from Averaged Operator Iterations}
\begin{abstract}
We study random iterations of averaged operators in Hilbert spaces
and prove that the associated residuals converge exponentially fast,
both in expectation and almost surely. Our results provide quantitative
bounds in terms of a single geometric parameter, giving sharp control
of convergence rates under minimal assumptions. As an application,
we introduce the concept of random nonlinear fusion frames. Here the
atoms are generated dynamically from the residuals of the iteration
and yield exact synthesis with frame-like stability in expectation.
We show that these frames achieve exponential sampling complexity
and encompass important special cases such as random projections and
randomized Kaczmarz methods. This reveals a link between stochastic
operator theory, frame theory, and randomized algorithms, and establishes
a structural tool for constructing nonlinear frame-like systems with
strong stability and convergence guarantees.
\end{abstract}

\author{James Tian}
\address{Mathematical Reviews, 535 W. William St, Suite 210, Ann Arbor, MI
48103, USA}
\email{jft@ams.org}
\maketitle

\section{Introduction}

Randomized iterative methods play a central role in modern analysis,
signal processing, and optimization. Classical examples include the
randomized Kaczmarz algorithm for solving linear systems, stochastic
gradient methods in optimization, and randomized projection methods
in frame theory. A common feature of these approaches is that convergence
guarantees are often stated in expectation, while almost sure quantitative
convergence rates remain more difficult to establish.

In this work we develop a general convergence theory for random iterations
of averaged operators on Hilbert spaces. Averaged operators provide
a unifying framework that includes firmly nonexpansive mappings, orthogonal
projections, resolvents of monotone operators, and their convex combinations.
They are central in convex optimization and monotone operator theory,
as surveyed in \cite{MR3616647}. Our focus is on random sequences
of such operators satisfying a mild independence and coercivity condition.
We prove that the residuals of the associated iteration decay exponentially
fast in expectation and, more strongly, almost surely with exponential
rate.

Here we present a pair of convergence theorems (Theorems \ref{thm:main}
and \ref{thm:b10}) giving sharp quantitative bounds for random averaged
iterations. The results are expressed in terms of a single geometric
parameter that measures the average coercivity of the random operators.
This parameter controls both the stability constants and the exponential
convergence rates, leading to clean and explicit estimates. 

As applications, we introduce the notion of random nonlinear fusion
frames. In contrast to classical frame systems, where the atoms are
fixed and linear, here the atoms are generated nonlinearly and randomly
from the residuals of the iteration. We show that these random nonlinear
fusion frames provide exact synthesis, frame-like stability inequalities
in expectation, and exponential sampling complexity. As special cases,
we recover random projection systems and randomized Kaczmarz methods,
thereby connecting our abstract theory with established algorithms
in harmonic analysis and numerical linear algebra.

Taken together, these results show that random averaged iterations
are not only analytically tractable, but also give rise to new frame-like
structures of independent mathematical interest. They establish a
link between stochastic operator theory, frame theory, and randomized
algorithms, leading to further developments in both analysis and applications.

\subsection{Related Work}

\textit{Averaged operators, resolvents, and proximal methods}. Averaged
(and firmly nonexpansive) operators provide a unifying language for
projections, resolvents of monotone operators, and many splitting
schemes; see \cite{MR3616647} for a comprehensive account and calculus
rules, and \cite{MR3299649} for sharp stability of compositions and
convex combinations of averaged maps. The proximal point algorithm
goes back to Rockafellar \cite{MR410483}, while operator-splitting
foundations include Lions-Mercier the Douglas-Rachford ADMM connections
analyzed by Eckstein-Bertsekas; see also the signal-processing survey
of Combettes-Pesquet \cite{MR551319,MR2858838,MR1168183}. In contrast,
our Theorems \ref{thm:main} and \ref{thm:b10} deliver explicit exponential
decay for random $\alpha$-averaged iterations under a mild coercivity-in-expectation
condition, together with a.s. exponential truncation, and we express
the constants via a single geometric parameter that simultaneously
controls stability and rates.

\textit{Randomized Kaczmarz and subspace actions}. The randomized
Kaczmarz method achieves expected exponential convergence for linear
systems, with sharp dimension and conditioning dependent bounds, and
has noisy-data variants with bias-variance tradeoffs \cite{MR2500924,MR2640019}.
Randomized projections for linear feasibility and extended Kaczmarz
least-squares variants further broaden this picture \cite{MR2724068,MR3069089}.
Randomized subspace actions connect these ideas to fusion-frame recovery
\cite{MR3439235}. 

Our results strictly generalizes these settings: projections and hyperplane
reflections are special cases of $\alpha$-averaged maps, and our
results yield a.s. exponential truncation and frame-like energy inequalities
in expectation for nonlinear, residual-driven atoms, which go beyond
linear static atoms in classical Kaczmarz/fusion-frame analyses.

\textit{Fusion frames and distributed representations}. Fusion frames
(frames of subspaces) were developed to model distributed sensing/processing,
with robust reconstruction from subspace projections and many design
tools \cite{MR2419707,MR2964018}. Operator-valued generalizations
include $g$-frames and frames of operators \cite{MR2239250,doi:10.1142/S0219691308002379}.
Prior work largely assumes linear, fixed atoms (subspace projections);
by contrast, our application shows that random nonlinear fusion frames
arise naturally from random averaged iterations: the atoms depend
nonlinearly on the evolving residual yet still provide exact synthesis
(a.s.), frame-like stability in expectation, and exponential sampling
complexity. This appears to be new within the fusion-frame literature. 

\textit{Stochastic coordinate and operator-splitting methods}. Randomized
block-coordinate and related stochastic splitting methods deliver
linear (geometric) rates under strong convexity/contractivity assumptions,
typically in expectation or with high probability bounds \cite{MR3179953}.
Our results differ in two ways: (i) we work at the level of random
averaged operators (resolvents, projections, averaged maps) with an
explicit operator-theoretic coercivity-in-mean parameter, and (ii)
we obtain almost-sure exponential truncation for the residual and
a synthesis interpretation that yields new frame-like objects.

The paper is organized as follows. \prettyref{sec:2} contains the
main assumptions and statements of the convergence theorems. \prettyref{sec:3}
discusses examples including random projections and hyperplane reflections.
\prettyref{sec:4} develops the application to random nonlinear fusion
frames.

\section{Random expansions via firmly nonexpansive operators}\label{sec:2}

Throughout this section, $H$ denotes a complex separable Hilbert
space with inner product linear in the second argument and conjugate-linear
in the first. Unless otherwise stated, mappings need not be linear.
We use standard tools such as Markov's inequality and the Borel-Cantelli
lemmas; see, e.g., \cite{MR3930614} or \cite{MR4226142}.

The following definitions are standard (see, e.g., \cite{MR1074005,MR3616647}).
\begin{defn}
A mapping $T\colon H\rightarrow H$ is:
\begin{enumerate}
\item nonexpansive if $\left\Vert Tx-Ty\right\Vert \leq\left\Vert x-y\right\Vert $
for all $x,y\in H$.
\item firmly nonexpansive (FNE) if 
\begin{equation}
\left\Vert Tx-Ty\right\Vert ^{2}\leq\left\langle Tx-Ty,x-y\right\rangle ,\quad x,y\in H.\label{eq:c1}
\end{equation}
\item $\alpha$-averaged $\left(0<\alpha\leq1\right)$ if there exists a
nonexpansive mapping $N\colon H\rightarrow H$ with 
\[
T=\left(1-\alpha\right)I+\alpha N.
\]
 
\end{enumerate}
\end{defn}

\begin{rem}
Every FNE map is nonexpansive. Moreover, if $T$ is FNE then $I-T$
is also FNE. Orthogonal projections are FNE; so are resolvents $(I+A)^{-1}$
of maximally monotone operators $A$. In the linear case, the FNE
condition coincides with $0\leq T\leq I$, i.e., $T$ is a positive
contraction. Here we do not require linearity.
\end{rem}

The following inequality is a direct consequence of the characterization
of $\alpha$-averaged operators in \cite[Prop. 4.35]{MR3616647}.
Since we require this one-point formulation in the arguments below,
and for the sake of completeness, we include a short self-contained
proof.
\begin{lem}
\label{lem:3-1}Let $H$ be a Hilbert space, let $N\colon H\to H$
be nonexpansive, and fix $\alpha\in(0,1]$. Define $T=(1-\alpha)I+\alpha N$
with $N\left(0\right)=0$. Then, for all $x\in H$,
\[
\left\Vert Tx\right\Vert ^{2}+\frac{1-\alpha}{\alpha}\left\Vert x-Tx\right\Vert ^{2}\leq\left\Vert x\right\Vert ^{2}.
\]
\end{lem}

\begin{proof}
For any $u,v\in H$ and $\alpha\in\left(0,1\right)$ we have the convex
combination identity
\[
\left\Vert \left(1-\alpha\right)u+\alpha v\right\Vert ^{2}=\left(1-\alpha\right)\left\Vert u\right\Vert ^{2}+\alpha\left\Vert v\right\Vert ^{2}-\alpha\left(1-\alpha\right)\left\Vert u-v\right\Vert ^{2}.
\]
Applying this with $u=x$, $v=Nx$ yields 
\[
\left\Vert Tx\right\Vert ^{2}=\left(1-\alpha\right)\left\Vert x\right\Vert ^{2}+\alpha\left\Vert Nx\right\Vert ^{2}-\alpha\left(1-\alpha\right)\left\Vert x-Nx\right\Vert ^{2}.
\]
Since $N$ is nonexpansive, $\left\Vert Nx\right\Vert =\left\Vert Nx-N\left(0\right)\right\Vert \leq\left\Vert x-0\right\Vert =\left\Vert x\right\Vert $,
hence 
\[
\left\Vert Tx\right\Vert ^{2}\leq\left\Vert x\right\Vert ^{2}-\alpha\left(1-\alpha\right)\left\Vert x-Nx\right\Vert ^{2}.
\]
Also, $x-Tx=\alpha\left(x-Nx\right)$, so that 
\[
\left\Vert x-Tx\right\Vert ^{2}=\alpha^{2}\left\Vert x-Nx\right\Vert ^{2}.
\]
Combining these give the stated inequality. 
\end{proof}
\begin{rem}
The parameter $\alpha$ determines the position of $\alpha$-averaged
operators between the identity and the full class of nonexpansive
mappings. The case $\alpha=\tfrac{1}{2}$ corresponds exactly to the
class of firmly nonexpansive mappings. In this case, \prettyref{lem:3-1}
specializes to the inequality
\[
\left\Vert Tx\right\Vert ^{2}+\left\Vert x-Tx\right\Vert ^{2}\leq\left\Vert x\right\Vert ^{2}.
\]
often called the firm Pythagoras inequality. This identity is classical
and underlies the analysis of orthogonal projections and of resolvents
of maximally monotone operators. If $\alpha=1$, then $T=N$ and the
notion of $\alpha$-averaged operator reduces to that of a general
nonexpansive mapping. 
\end{rem}

\begin{defn}
Let $\left(\Omega,\mathcal{F},\mathbb{P}\right)$ be a probability
space. A mapping $T\colon\Omega\times H\to H$ is a random operator
if for each fixed $x\in H$, the section $\omega\mapsto T\left(\omega,x\right)$
is $\mathcal{F}/\mathcal{B}(H)$ measurable, and for $\mathbb{P}$-a.e.
$\omega$, the section $x\mapsto T\left(\omega,x\right)$ belongs
to a prescribed class (e.g., continuous, nonexpansive, $\alpha$-averaged,
etc.). When, in addition, $x\mapsto T\left(\omega,x\right)$ is continuous
for a.e. $\omega$ (Caratheodory condition), then $\left(\omega,x\right)\mapsto T\left(\omega,x\right)$
is jointly measurable.
\end{defn}

In what follows we fix an $\alpha\in\left(0,1\right)$ and assume:
\begin{itemize}
\item[(A1)]  (i.i.d. $\alpha$-averaged maps) There exists a random operator
$T\colon\Omega\times H\to H$ such that for $\mathbb{P}$-a.e. $\omega$,
$T(\omega,\cdot)$ is $\alpha$-averaged. Let $\{T_{k}\}_{k\ge1}$
be i.i.d. copies of $T$, independent across $k$.
\item[(A2)]  (Uniform mean-square coercivity) There exists a constant $C\in\left(0,1\right)$
such that
\[
\mathbb{E}\left[\left\Vert T\left(\omega,u\right)\right\Vert ^{2}\right]\ge C\left\Vert u\right\Vert ^{2}\qquad\text{for all }u\in H,
\]
where the expectation is with respect to $\omega\sim\mathbb{P}$.
Equivalently, for each $u\in H$, $\omega\mapsto\left\Vert T\left(\omega,u\right)\right\Vert ^{2}$
is integrable and the above lower bound holds.
\item[(A3)]  (One-step measurability) For each $x\in H$, the random variables
$T_{k}\left(x\right)$ are square-integrable and $\sigma\left(T_{1},\dots,T_{k-1}\right)$-measurable
conditional expectations in the proofs below are well-defined.
\end{itemize}
\begin{rem}
The endpoint $\alpha=1$ corresponds to $T=N$ nonexpansive. Our convergence
theorem below uses the $\alpha$-averaged inequality in a way that
requires $\alpha<1$; for $\alpha=1$ one needs a different (stronger)
strict-contraction-in-mean assumption (e.g., $\mathbb{E}[\left\Vert N\left(\omega,u\right)\right\Vert ^{2}]\le\left(1-\eta\right)\left\Vert u\right\Vert ^{2}$
for some $\eta>0$) to conclude decay of residuals.
\end{rem}

\textbf{Nonlinear telescoping decomposition}. Fix $x\in H$ and define
recursively, for $n\ge1$,
\[
R_{0}\left(x\right)\coloneqq x,\qquad F_{n}\left(x\right)\coloneqq T_{n}\left(R_{n-1}\left(x\right)\right),\qquad R_{n}\left(x\right)\coloneqq R_{n-1}\left(x\right)-F_{n}\left(x\right).
\]
Then for each $n\ge1$ the algebraic identity
\[
R_{n-1}\left(x\right)=F_{n}\left(x\right)+R_{n}\left(x\right)
\]
holds by definition. Summing these identities for $k=1,\dots,n$ yields
the telescoping decomposition
\begin{equation}
x=\sum\nolimits^{n}_{k=1}F_{k}\left(x\right)+R_{n}\left(x\right),\quad n\ge1.\label{eq:tele}
\end{equation}
This is a purely pointwise identity in $H$ and does not require linearity
of the $T_{k}$.

Set

\[
\rho_{\alpha}\left(C\right)\coloneqq\frac{\alpha}{1-\alpha}\left(1-C\right).
\]
Note that $0<\rho_{\alpha}\left(C\right)<1$ holds automatically for
all $C\in\left(0,1\right)$ if $\alpha\le\tfrac{1}{2}$; for $\alpha\in\left(\tfrac{1}{2},1\right)$
it requires $\tfrac{2\alpha-1}{\alpha}<C<1$.
\begin{thm}[Random iteration for $\alpha$-averaged maps]
\label{thm:main} Assume (A1)-(A3) with $\alpha\in\left(0,1\right)$
and $C\in(0,1)$, and define $\left\{ R_{n}\left(x\right)\right\} _{n\geq0}$,
$\left\{ F_{n}\left(x\right)\right\} _{n\geq1}$ by the nonlinear
recursion above. Suppose moreover that $\rho_{\alpha}\left(C\right)<1$.
Then for every $x\in H$:
\begin{enumerate}
\item \label{enu:1}For each $n\ge0$, 
\[
\mathbb{E}\left[\left\Vert R_{n}\left(x\right)\right\Vert ^{2}\right]\le\rho_{\alpha}\left(C\right)^{n}\left\Vert x\right\Vert ^{2}\quad\left(n\ge0\right).
\]
\item \label{enu:2}$R_{n}\left(x\right)\to0$ almost surely as $n\to\infty$.
\item \label{enu:3}The decomposition
\[
x=\sum\nolimits^{\infty}_{k=1}F_{k}(x)
\]
holds almost surely in the strong topology. Moreover,
\[
\lim_{n\to\infty}\mathbb{E}\left\Vert x-\sum\nolimits^{n}_{k=1}F_{k}(x)\right\Vert ^{2}=0.
\]
\item \label{enu:4}The sequence of partial sums satisfies the inequalities
\[
C\left\Vert x\right\Vert ^{2}\le\lim_{n\to\infty}\mathbb{E}\left[\sum\nolimits^{n}_{k=1}\left\Vert F_{k}\left(x\right)\right\Vert ^{2}\right]\le U_{\alpha}\|x\|^{2}
\]
where 
\[
U_{\alpha}=\begin{cases}
1, & \alpha\le\tfrac{1}{2},\\[1ex]
1+\frac{2\alpha-1}{\alpha}\cdot\dfrac{\rho_{\alpha}\left(C\right)}{1-\rho_{\alpha}\left(C\right)}, & \alpha>\tfrac{1}{2}.
\end{cases}
\]
\end{enumerate}
\end{thm}

\begin{proof}
Fix $x\in H$. By \prettyref{lem:3-1} applied to the (random) map
$T_{n}\left(\cdot\right)$ at the point $R_{n-1}$,
\begin{equation}
\left\Vert T_{n}\left(R_{n-1}\right)\right\Vert ^{2}+\frac{1-\alpha}{\alpha}\left\Vert R_{n}\right\Vert ^{2}\le\left\Vert R_{n-1}\right\Vert ^{2}\quad\text{a.s.}\label{eq:b3}
\end{equation}

\eqref{enu:1} Taking conditional expectation $\mathbb{E}_{n-1}\left[\cdot\right]\coloneqq\mathbb{E}\left[\cdot\mid T_{1},\dots,T_{n-1}\right]$
in \eqref{eq:b3} and using nonnegativity yields
\[
\frac{1-\alpha}{\alpha}\mathbb{E}_{n-1}\left\Vert R_{n}\right\Vert ^{2}\le\left\Vert R_{n-1}\right\Vert ^{2}-\mathbb{E}_{n-1}\left\Vert T_{n}\left(R_{n-1}\right)\right\Vert ^{2}.
\]
By independence of $T_{n}$ from $\left(T_{1},\dots,T_{n-1}\right)$
and the coercivity (A2),
\[
\mathbb{E}_{n-1}\left\Vert T_{n}\left(R_{n-1}\right)\right\Vert ^{2}\ge C\left\Vert R_{n-1}\right\Vert ^{2}.
\]
Hence,
\[
\mathbb{E}_{n-1}\left\Vert R_{n}\right\Vert ^{2}\le\frac{\alpha}{1-\alpha}\left(1-C\right)\left\Vert R_{n-1}\right\Vert ^{2}=\rho_{\alpha}\left(C\right)\left\Vert R_{n-1}\right\Vert ^{2}.
\]
Taking full expectation and iterating gives

\[
\mathbb{E}\left\Vert R_{n}\right\Vert ^{2}\le\rho_{\alpha}\left(C\right)\mathbb{E}\left\Vert R_{n-1}\right\Vert ^{2}\le\cdots\le\rho_{\alpha}\left(C\right)^{n}\left\Vert x\right\Vert ^{2},
\]
as claimed.

\eqref{enu:2} Fix $\delta>0$. By Markov's inequality and \eqref{enu:1},
\[
\mathbb{P}\left(\left\Vert R_{n}\right\Vert >\delta\right)\le\frac{\mathbb{E}\left\Vert R_{n}\right\Vert ^{2}}{\delta^{2}}\le\frac{\rho_{\alpha}\left(C\right)^{n}}{\delta^{2}}\left\Vert x\right\Vert ^{2}.
\]
Since $\sum_{n}\mathbb{P}\left(\left\Vert R_{n}\right\Vert >\delta\right)<\infty$,
the Borel-Cantelli lemma yields $\left\Vert R_{n}\right\Vert \to0$
a.s.

\eqref{enu:3} The telescoping identity \eqref{eq:tele} holds for
every $n$. By \eqref{enu:2} $R_{n}\to0$ a.s., hence $x=\sum^{\infty}_{k=1}F_{k}\left(x\right)$
a.s. in the strong topology. Moreover, by \eqref{enu:1},
\[
\mathbb{E}\left\Vert x-\sum\nolimits^{n}_{k=1}F_{k}\left(x\right)\right\Vert ^{2}=\mathbb{E}\left\Vert R_{n}\right\Vert ^{2}\xrightarrow[n\to\infty]{}0.
\]

\eqref{enu:4} First, taking expectations in \eqref{eq:b3} and summing
over $n=1,\dots,N$ gives
\begin{align*}
\sum\nolimits^{N}_{n=1}\mathbb{E}\left\Vert F_{n}\right\Vert ^{2}+\frac{1-\alpha}{\alpha}\sum\nolimits^{N}_{n=1}\mathbb{E}\left\Vert R_{n}\right\Vert ^{2} & \le\sum\nolimits^{N}_{n=1}\mathbb{E}\left\Vert R_{n-1}\right\Vert ^{2}\\
 & =\left\Vert x\right\Vert ^{2}+\sum\nolimits^{N-1}_{n=1}\mathbb{E}\left\Vert R_{n}\right\Vert ^{2}.
\end{align*}
Rearranging,
\begin{equation}
\sum\nolimits^{N}_{n=1}\mathbb{E}\left\Vert F_{n}\right\Vert ^{2}+\Big(\frac{1-2\alpha}{\alpha}\Big)\sum\nolimits^{N-1}_{n=1}\mathbb{E}\left\Vert R_{n}\right\Vert ^{2}+\frac{1-\alpha}{\alpha}\mathbb{E}\left\Vert R_{N}\right\Vert ^{2}\le\left\Vert x\right\Vert ^{2}.\label{eq:b4}
\end{equation}

If $\alpha\in\left(0,1/2\right)$, then the second and third terms
on the left of \eqref{eq:b4} are nonnegative. Dropping them gives,
for all $N$,
\[
\sum\nolimits^{N}_{n=1}\mathbb{E}\left\Vert F_{n}\right\Vert ^{2}\le\left\Vert x\right\Vert ^{2}
\]
hence by monotone convergence,
\[
\lim_{N\to\infty}\mathbb{E}\left[\sum\nolimits^{N}_{n=1}\left\Vert F_{n}(x)\right\Vert ^{2}\right]\le\left\Vert x\right\Vert ^{2}.
\]
For the lower bound, note that, by (A2),
\[
\mathbb{E}\left\Vert F_{1}\left(x\right)\right\Vert ^{2}=\mathbb{E}\left\Vert T_{1}(x)\right\Vert ^{2}\ge C\left\Vert x\right\Vert ^{2},
\]
so the nondecreasing limit satisfies
\[
\lim_{N\to\infty}\mathbb{E}\left[\sum\nolimits^{N}_{n=1}\left\Vert F_{n}\left(x\right)\right\Vert ^{2}\right]\ge\mathbb{E}\left\Vert F_{1}\left(x\right)\right\Vert ^{2}\ge C\left\Vert x\right\Vert ^{2}.
\]

The case when $\alpha>1/2$. Now the middle term in the left-hand
side of \eqref{eq:b4} need not be nonnegative and cannot be discarded. 

Using the decay from part \eqref{enu:1}, $\mathbb{E}\left\Vert R_{n}\right\Vert ^{2}\le\rho_{\alpha}\left(C\right)^{n}\left\Vert x\right\Vert ^{2}$
with $\rho_{\alpha}\left(C\right)<1$, we bound the (potentially negative)
middle term:
\[
\left(\frac{1-2\alpha}{\alpha}\right)\sum\nolimits^{N-1}_{n=1}\mathbb{E}\left\Vert R_{n}\right\Vert ^{2}\ge-\left(\frac{2\alpha-1}{\alpha}\right)\frac{\rho_{\alpha}\left(C\right)}{1-\rho_{\alpha}\left(C\right)}\left\Vert x\right\Vert ^{2}.
\]
Hence from \eqref{eq:b4} we obtain, for all $N$,
\[
\sum\nolimits^{N}_{n=1}\mathbb{E}\left\Vert F_{n}\right\Vert ^{2}\le\left\Vert x\right\Vert ^{2}+\frac{2\alpha-1}{\alpha}\frac{\rho_{\alpha}\left(C\right)}{1-\rho_{\alpha}\left(C\right)}\left\Vert x\right\Vert ^{2}.
\]
Letting $N\to\infty$ and using monotone convergence,
\[
\lim_{N\to\infty}\mathbb{E}\left[\sum\nolimits^{N}_{n=1}\left\Vert F_{n}\left(x\right)\right\Vert ^{2}\right]\le\left(1+\frac{2\alpha-1}{\alpha}\frac{\rho_{\alpha}\left(C\right)}{1-\rho_{\alpha}\left(C\right)}\right)\left\Vert x\right\Vert ^{2}.
\]

The same lower bound $\ge C\left\Vert x\right\Vert ^{2}$ holds as
above from the $n=1$ term.
\end{proof}
\begin{cor}[Firmly nonexpansive case $\alpha=\tfrac{1}{2}$]
Assume (A1)-(A3) with $\alpha=\tfrac{1}{2}$ and $C\in\left(0,1\right)$,
and define $\left\{ R_{n}\left(x\right)\right\} _{n\geq0}$ and $\left\{ F_{n}\left(x\right)\right\} _{n\geq1}$
by the nonlinear recursion above. Then for every $x\in H$:
\begin{enumerate}
\item ${\displaystyle \mathbb{E}\left\Vert R_{n}\left(x\right)\right\Vert ^{2}\le\left(1-C\right)^{n}\left\Vert x\right\Vert ^{2}}$,
$n\ge0$.
\item $R_{n}\left(x\right)\to0$ almost surely as $n\to\infty$.
\item ${\displaystyle x=\sum\nolimits^{\infty}_{k=1}F_{k}\left(x\right)}$
almost surely in the strong topology, and
\item ${\displaystyle \lim_{n\to\infty}\mathbb{E}\Big\| x-\sum\nolimits^{n}_{k=1}F_{k}\left(x\right)\Big\|^{2}=0.}$
\item ${\displaystyle C\left\Vert x\right\Vert ^{2}\le\lim_{n\to\infty}\mathbb{E}\Big[\sum\nolimits^{n}_{k=1}\left\Vert F_{k}\left(x\right)\right\Vert ^{2}\Big]\le\left\Vert x\right\Vert ^{2}.}$
\end{enumerate}
\end{cor}

\begin{proof}
Apply \prettyref{thm:main} with $\alpha=\tfrac{1}{2}$. Then $\rho_{\alpha}\left(C\right)=\rho_{1/2}\left(C\right)=1-C<1$,
so the standing hypothesis $\rho_{\alpha}\left(C\right)<1$ holds
automatically. In item (4) the upper constant reduces to $U_{1/2}=1$,
giving the upper bound.
\end{proof}
\begin{rem}[Comparison with the linear positive-contraction case]
 When each $T_{k}$ is linear and FNE with $\alpha=\tfrac{1}{2}$
(i.e., a positive contraction), \prettyref{lem:3-1} reduces to the
inequality
\[
\left\Vert Tx\right\Vert ^{2}+\left\Vert x-Tx\right\Vert ^{2}\le\left\Vert x\right\Vert ^{2}
\]
which is the key estimate used in the linear theory \cite{tian2025randomoperatorvaluedframeshilbert}.
The proof above shows that the entire iterative argument carries over
verbatim to the nonlinear $\alpha$-averaged setting once \prettyref{lem:3-1}
and (A2) are in place.
\end{rem}

\begin{thm}[Almost-sure exponential rate]
\label{thm:b10}Assume (A1)-(A3) with $\alpha\in(0,1)$ and $C\in(0,1)$,
and suppose $\rho_{\alpha}(C)<1$. Fix $x\in H$ and set 
\[
\gamma\coloneqq-\tfrac{1}{2}\log\rho_{\alpha}\left(C\right)>0.
\]
Then for every $\ensuremath{\varepsilon\in(0,\gamma)}$ there exists
a finite (random) index $\ensuremath{N=N(\omega,x,\varepsilon)}$
such that
\[
\left\Vert R_{n}\left(x\right)\right\Vert \le e^{-\left(\gamma-\varepsilon\right)n}\left\Vert x\right\Vert \qquad\text{for all }n\ge N,\quad\text{\ensuremath{\mathbb{P}}-a.s.}
\]
Equivalently,
\[
\limsup_{n\to\infty}\frac{1}{n}\log\left\Vert R_{n}\left(x\right)\right\Vert \le-\gamma\qquad\text{\ensuremath{\mathbb{P}}-a.s.}
\]
\end{thm}

\begin{proof}
By \prettyref{thm:main} \eqref{enu:1} we have, for every $n\ge0$,
\begin{equation}
\mathbb{E}\left\Vert R_{n}(x)\right\Vert ^{2}\le\rho_{\alpha}\left(C\right)^{n}\left\Vert x\right\Vert ^{2},\label{eq:L2-decay-new}
\end{equation}
with $0<\rho_{\alpha}\left(C\right)<1$ by hypothesis. Fix any $\varepsilon\in\left(0,\gamma\right)$
and define
\[
\theta\coloneqq e^{-\left(\gamma-\varepsilon\right)}\in(\sqrt{\rho_{\alpha}\left(C\right)},1),\qquad A_{n}\coloneqq\left\{ \left\Vert R_{n}\left(x\right)\right\Vert >\theta^{n}\left\Vert x\right\Vert \right\} .
\]
(Note that $\theta>\sqrt{\rho_{\alpha}(C)}$ holds because $\log\theta=-\left(\gamma-\varepsilon\right)=\tfrac{1}{2}\log\rho_{\alpha}\left(C\right)+\varepsilon$.)

Apply Markov’s inequality to the nonnegative random variable $\left\Vert R_{n}\left(x\right)\right\Vert ^{2}$:
\[
\mathbb{P}\left(A_{n}\right)=\mathbb{P}\left(\left\Vert R_{n}(x)\right\Vert ^{2}>\theta^{2n}\left\Vert x\right\Vert ^{2}\right)\le\frac{\mathbb{E}\left\Vert R_{n}(x)\right\Vert ^{2}}{\theta^{2n}\left\Vert x\right\Vert ^{2}}.
\]
Using \eqref{eq:L2-decay-new},
\[
\mathbb{P}\left(A_{n}\right)\le\left(\frac{\rho_{\alpha}\left(C\right)}{\theta^{2}}\right)^{n}.
\]
Since $\rho_{\alpha}\left(C\right)/\theta^{2}\in\left(0,1\right)$,
the geometric series converges:
\[
\sum^{\infty}_{n=0}\mathbb{P}\left(A_{n}\right)\le\sum\nolimits^{\infty}_{n=0}\left(\frac{\rho_{\alpha}\left(C\right)}{\theta^{2}}\right)^{n}<\infty.
\]
By the Borel-Cantelli lemma, $\mathbb{P}\left(A_{n}\ \text{i.o.}\right)=0$.
Equivalently, there exists a finite (random) $N=N(\omega,x,\varepsilon)$
such that $\left\Vert R_{n}(x)\right\Vert \le\theta^{n}\left\Vert x\right\Vert $
for all $n\ge N$, almost surely. Since $\theta=e^{-(\gamma-\varepsilon)}$,
this gives the stated bound.

For the limsup, note that the above holds for every $\varepsilon\in\left(0,\gamma\right)$,
hence
\[
\limsup_{n\to\infty}\frac{1}{n}\log\left\Vert R_{n}\left(x\right)\right\Vert \le\log\theta=-(\gamma-\varepsilon)\quad a.s.
\]
Letting $\varepsilon\downarrow0$ yields the claimed $\limsup\le-\gamma$
almost surely.
\end{proof}
\begin{cor}
Under the hypotheses of \prettyref{thm:b10}, let $\gamma=-\tfrac{1}{2}\log\rho_{\alpha}\left(C\right)>0$.
Then for every $\varepsilon\in\left(0,\gamma\right)$ there exists
a finite (random) $N$ such that
\[
\left\Vert x-\sum\nolimits^{n}_{k=1}F_{k}\left(x\right)\right\Vert =\left\Vert R_{n}\left(x\right)\right\Vert \le e^{-\left(\gamma-\varepsilon\right)n}\left\Vert x\right\Vert ,\;\text{for all }n\ge N,\quad\text{\ensuremath{\mathbb{P}}-a.s.}
\]
In particular,
\[
\limsup_{n\to\infty}\frac{1}{n}\log\left\Vert x-\sum\nolimits^{n}_{k=1}F_{k}\left(x\right)\right\Vert \le-\gamma,\;\text{\ensuremath{\mathbb{P}}-a.s.}
\]
\end{cor}

\begin{proof}
The telescoping identity yields $x-\sum^{n}_{k=1}F_{k}\left(x\right)=R_{n}\left(x\right)$
for every $n$. Thus the corollary is an immediate restatement of
\prettyref{thm:b10} with $\left\Vert R_{n}\left(x\right)\right\Vert $
replaced by the partial-sum error.
\end{proof}
\begin{rem}[On the $\varepsilon$ in the exponential rate]
 The proof of \prettyref{thm:b10} relies only on the $L^{2}$ decay
$\mathbb{E}\left\Vert R_{n}(x)\right\Vert ^{2}\le\rho_{\alpha}\left(C\right)^{n}\left\Vert x\right\Vert ^{2},$
and on a Markov-Borel-Cantelli argument. This yields an eventual almost-sure
bound with any base $\theta\in(\sqrt{\rho_{\alpha}(C)},1)$, which
translates to
\[
\left\Vert R_{n}\left(x\right)\right\Vert \le e^{-\left(\gamma-\varepsilon\right)n}\left\Vert x\right\Vert \quad\text{for large \ensuremath{n}, a.s.,}
\]
for every $\varepsilon\in\left(0,\gamma\right)$. Hence the conclusion
$\limsup_{n\to\infty}\tfrac{1}{n}\log\left\Vert R_{n}\left(x\right)\right\Vert \le-\gamma$.
To remove the $\varepsilon$ and obtain the exact rate 
\[
\left\Vert R_{n}\left(x\right)\right\Vert \le e^{-\gamma n}\left\Vert x\right\Vert 
\]
eventually a.s., one typically needs stronger probabilistic control,
such as exponential moment bounds for $\left\Vert R_{n}\left(x\right)\right\Vert $,
or concentration inequalities obtained from supermartingale methods
(e.g., Azuma-Hoeffding or Bernstein inequalities). Since such assumptions
are not part of (A1)-(A3), the $\varepsilon$-slack is the sharp conclusion
available under the present hypotheses.
\end{rem}

\section{Examples verifying (A2)}\label{sec:3}

Recall (A2): there exists $C\in(0,1)$ such that
\[
\mathbb{E}\left[\left\Vert T\left(\omega\right)u\right\Vert ^{2}\right]\ge C\left\Vert u\right\Vert ^{2}\quad\text{for all }u\in H.
\]
We present three classes where (A2) holds with an explicit $C$.
\begin{example}[Random orthogonal projections]
 Let $V\left(\omega\right)\subset H$ be a (measurable) random closed
subspace and let $T\left(\omega\right)=P_{V\left(\omega\right)}$
be the orthogonal projection onto $V\left(\omega\right)$. Then
\begin{equation}
\left\Vert P_{V\left(\omega\right)}u\right\Vert ^{2}=\left\langle u,P_{V\left(\omega\right)}u\right\rangle .\label{eq:d1}
\end{equation}
Define the (bounded, selfadjoint, positive) operator
\[
G\coloneqq\mathbb{E}\big[P_{V\left(\omega\right)}\big]\in\mathcal{B}\left(H\right).
\]
By Fubini/Tonelli,
\[
\mathbb{E}\left[\left\Vert T\left(\omega\right)u\right\Vert ^{2}\right]=\mathbb{E}\left[\left\langle u,P_{V\left(\omega\right)}u\right\rangle \right]=\left\langle u,Gu\right\rangle .
\]
Hence (A2) holds provided $G\geq CI$, i.e.,
\[
\left\langle u,Gu\right\rangle \ge C\left\Vert u\right\Vert ^{2}\quad\forall u\in H.
\]
Equivalently, $C=\lambda_{\min}(G)>0$. In finite dimensions this
is the usual ``uniform coverage of directions'' condition; in infinite
dimensions one may assume $G$ has a spectral gap at $0$ on the relevant
closed span.
\end{example}

\begin{example}[Randomized Kaczmarz / random hyperplanes through the origin]
\label{exa:4-2}Let $a\colon\Omega\to H\setminus\{0\}$ be a random
direction and define the orthogonal projection onto the hyperplane
$\left\{ x\in H:\langle a(\omega),x\rangle=0\right\} $ by
\[
P_{\omega}\coloneqq I-\frac{a\left(\omega\right)\otimes a\left(\omega\right)}{\left\Vert a(\omega)\right\Vert ^{2}}.
\]
Set $T\left(\omega\right)=P_{\omega}$. Then, for any $u\in H$,
\[
\left\Vert P_{\omega}u\right\Vert ^{2}=\left\Vert u\right\Vert ^{2}-\frac{\left|\left\langle a\left(\omega\right),u\right\rangle \right|^{2}}{\left\Vert a\left(\omega\right)\right\Vert ^{2}}.
\]
Taking expectations,
\[
\mathbb{E}\left\Vert T\left(\omega\right)u\right\Vert ^{2}=\left\Vert u\right\Vert ^{2}-\left\langle \Sigma u,u\right\rangle ,\quad\Sigma\coloneqq\mathbb{E}\left[\frac{a\left(\omega\right)\otimes a\left(\omega\right)}{\left\Vert a\left(\omega\right)\right\Vert ^{2}}\right].
\]
Therefore,
\[
\mathbb{E}\left\Vert T\left(\omega\right)u\right\Vert ^{2}=\left\langle \left(I-\Sigma\right)u,u\right\rangle ,
\]
and (A2) holds with
\end{example}

\[
C=\lambda_{\min}\left(I-\Sigma\right)=1-\lambda_{\max}\left(\Sigma\right).
\]
Thus it suffices that the normalized second-moment operator $\Sigma$
satisfy $\lambda_{\max}(\Sigma)<1$, which expresses that the random
directions are not concentrated on a strict subspace. In $\mathbb{R}^{d}$,
for instance, if $a/\left\Vert a\right\Vert $ is isotropic then $\Sigma=\tfrac{1}{d}I$
and $C=1-1/d$. 
\begin{example}[Random averaged projections]
 Let $V\left(\omega\right)$ be as in \eqref{eq:d1} and fix $\alpha\in(0,1]$.
Consider the $\alpha$-averaged operator
\[
T(\omega)\coloneqq\left(1-\alpha\right)I+\alpha P_{V\left(\omega\right)}.
\]
Since $P_{V\left(\omega\right)}$ is an orthogonal projection (selfadjoint,
idempotent), for any $u\in H$,
\begin{align*}
\left\Vert T\left(\omega\right)u\right\Vert ^{2} & =\left\Vert \left(1-\alpha\right)u+\alpha P_{V\left(\omega\right)}u\right\Vert ^{2}\\
 & =\left(1-\alpha\right)^{2}\left\Vert u\right\Vert ^{2}+2\alpha\left(1-\alpha\right)\left\langle u,P_{V\left(\omega\right)}u\right\rangle +\alpha^{2}\left\Vert P_{V\left(\omega\right)}u\right\Vert ^{2}\\
 & =\left(1-\alpha\right)^{2}\left\Vert u\right\Vert ^{2}+\left(2\alpha\left(1-\alpha\right)+\alpha^{2}\right)\left\Vert P_{V\left(\omega\right)}u\right\Vert ^{2}\\
 & =\left(1-\alpha\right)^{2}\left\Vert u\right\Vert ^{2}+\left(2\alpha-\alpha^{2}\right)\left\Vert P_{V\left(\omega\right)}u\right\Vert ^{2}.
\end{align*}
Taking expectations and writing $G=\mathbb{E}[P_{V\left(\omega\right)}]$
as before,
\[
\mathbb{E}\left\Vert T\left(\omega\right)u\right\Vert ^{2}=\left(1-\alpha\right)^{2}\left\Vert u\right\Vert ^{2}+\left(2\alpha-\alpha^{2}\right)\left\langle u,Gu\right\rangle .
\]

If $G\geq\gamma I$ with some $\gamma\in(0,1]$, then
\[
\mathbb{E}\left\Vert T\left(\omega\right)u\right\Vert ^{2}\ge\left(\left(1-\alpha\right)^{2}+\left(2\alpha-\alpha^{2}\right)\gamma\right)\left\Vert u\right\Vert ^{2},
\]
so (A2) holds with the explicit constant
\[
C=\left(1-\alpha\right)^{2}+\left(2\alpha-\alpha^{2}\right)\lambda_{\min}\left(G\right).
\]
Note that when $\alpha=1/2$ (firm averaging), 
\[
C=\frac{1}{4}+\frac{3}{4}\lambda_{\min}\left(G\right).
\]
\end{example}

\begin{rem}
In the above examples the verification of (A2) reduces to a linear
spectral condition on the positive operator $G=\mathbb{E}[P_{V(\omega)}]$
(or, equivalently, on $\Sigma$ in \prettyref{exa:4-2}). This is
robust and easy to check in applications.

For more general random averaged operators of the form $T\left(\omega\right)=\left(1-\alpha\right)I+\alpha N\left(\omega\right)$
with $N\left(\omega\right)$ nonexpansive and $N\left(\omega\right)\left(0\right)=0$,
one cannot expect a universal lower bound of $\mathbb{E}\left\Vert T\left(\omega\right)u\right\Vert ^{2}$
without additional structure. The projection-based cases above provide
a broad and practically important class (including randomized coordinate
projections and Kaczmarz-type schemes) where (A2) is guaranteed.

In finite dimensions, if the random subspaces $V\left(\omega\right)$
are drawn from a distribution with a density that is uniformly nondegenerate
on the Grassmannian (so that all directions are covered with positive
probability), then $\lambda_{\min}\left(G\right)>0$ and (A2) holds.
In infinite dimensions, the same conclusion holds on the closed span
of the union $\overline{span}\bigcup_{\omega}V(\omega)$, provided
$G$ has a spectral gap at $0$ on that subspace.
\end{rem}

\section{Random nonlinear fusion frames with exponential sampling rates}\label{sec:4}

In this section, we show how Theorems \ref{thm:main} and \ref{thm:b10}
give rise to a new class of frame-like structures in Hilbert spaces.
Classical frame theory concerns families $\left\{ f_{i}\right\} $
of vectors satisfying the inequality 
\[
A\left\Vert x\right\Vert ^{2}\le\sum\nolimits_{i}\left|\left\langle f_{i},x\right\rangle \right|^{2}\le B\left\Vert x\right\Vert ^{2},
\]
and fusion frame theory extends this to collections of subspaces.
Recall that a family $\left\{ \left(W_{i},v_{i}\right)\right\} _{i\in I}$
of closed subspaces $W_{i}\subset H$ with weights $v_{i}>0$ is a
fusion frame if there exist constants $0<A\le B<\infty$ such that
\[
A\left\Vert x\right\Vert ^{2}\le\sum\nolimits_{i\in I}v^{2}_{i}\left\Vert P_{W_{i}}x\right\Vert ^{2}\le B\left\Vert x\right\Vert ^{2}
\]
for all $x\in H$, where $P_{W_{i}}$ denotes the orthogonal projection
onto $W_{i}$. Equivalently, the fusion frame operator $S\coloneqq\sum_{i\in I}v^{2}_{i}P_{W_{i}}$
is bounded and invertible. The special case $S=I$ is a Parseval fusion
frame.

Our results allow one to define a random nonlinear fusion frame (RNFF),
where the atoms arise nonlinearly and randomly through the iterative
scheme studied in \prettyref{sec:2}. This provides a conceptual application
of the convergence theory developed above.

Let $H$ be a separable Hilbert space, and assume (A1)-(A3) with fixed
$\alpha\in\left(0,1\right)$ and $C\in\left(0,1\right)$. For each
$x\in H$ define recursively
\[
R_{0}\left(x\right)\coloneqq x,\quad F_{n}\left(x\right)\coloneqq T_{n}\left(R_{n-1}\left(x\right)\right),\quad R_{n}\left(x\right)\coloneqq R_{n-1}\left(x\right)-F_{n}\left(x\right).
\]
Thus each atom $F_{n}\left(x\right)$ is obtained by applying a random
operator to the current residual. The sequence $F=\left\{ F_{n}\right\} _{n\geq1}$
depends on $x$ nonlinearly and randomly, yet as we show below, it
provides a complete and stable representation of $x$.
\begin{defn}
We say that $F=\left\{ F_{n}\right\} $ is a random nonlinear fusion
frame (RNFF) for $H$ if for each $x\in H$, 
\begin{enumerate}
\item (Exact synthesis) $x=\sum^{\infty}_{n=1}F_{n}\left(x\right)$ almost
surely in $H$;
\item (Frame bounds in expectation) there exist $0<A\le B<\infty$ such
that
\[
A\left\Vert x\right\Vert ^{2}\leq\lim_{N\to\infty}\mathbb{E}\left[\sum\nolimits^{N}_{n=1}\left\Vert F_{n}(x)\right\Vert ^{2}\right]\le B\left\Vert x\right\Vert ^{2},\quad x\in H.
\]
\end{enumerate}
\end{defn}

\begin{thm}[RNFF from $\alpha$-averaged iterations]
\label{thm:RNFF} Assume (A1)-(A3) with $\alpha\in\left(0,1\right)$
and $C\in\left(0,1\right)$. Let $\rho_{\alpha}\left(C\right)=\frac{\alpha}{1-\alpha}\left(1-C\right)<1$
and $\gamma=-\tfrac{1}{2}\log\rho_{\alpha}\left(C\right)>0$. Then
$F=\{F_{n}\}$ is an RNFF with the following properties:
\begin{enumerate}
\item (Exact synthesis and bounds) Almost surely,
\[
x=\sum^{\infty}_{n=1}F_{n}(x)\;\text{in }H,
\]
and for all $x\in H$,
\[
C\left\Vert x\right\Vert ^{2}\le\lim_{N\to\infty}\mathbb{E}\left[\sum\nolimits^{N}_{n=1}\left\Vert F_{n}\left(x\right)\right\Vert ^{2}\right]\le U_{\alpha}\left\Vert x\right\Vert ^{2},
\]
where
\[
U_{\alpha}=\begin{cases}
1, & \alpha\le\tfrac{1}{2},\\[2mm]
1+\dfrac{2\alpha-1}{\alpha}\cdot\dfrac{\rho_{\alpha}(C)}{1-\rho_{\alpha}(C)}, & \alpha>\tfrac{1}{2}.
\end{cases}
\]
\item (Exponential sampling rate) For every $\varepsilon\in\left(0,\gamma\right)$
and every $x\in H$, there exists a random integer $N\left(\omega,x,\varepsilon\right)$
such that
\begin{equation}
\left\Vert x-\sum\nolimits^{N}_{n=1}F_{n}\left(x\right)\right\Vert =\left\Vert R_{N}\left(x\right)\right\Vert \le e^{-\left(\gamma-\varepsilon\right)N}\left\Vert x\right\Vert \label{eq:d-1}
\end{equation}
for all $N\ge N\left(\omega,x,\varepsilon\right)$, a.s. 
\end{enumerate}
\end{thm}

\begin{proof}
Part (1) is an immediate reformulation of \prettyref{thm:main} \eqref{enu:3}-\eqref{enu:4}.
Part (2) is precisely \prettyref{thm:b10}, rephrased in terms of
the synthesis operator $S_{N}\left(x\right)=\sum^{N}_{n=1}F_{n}\left(x\right)$. 
\end{proof}
\begin{cor}[Firmly nonexpansive case]
 If $\alpha=\tfrac{1}{2}$, then $F$ is an expected Parseval RNFF
with lower frame constant $A=C$ and upper constant $B=1$. Moreover,
\[
\left\Vert x-\sum\nolimits^{N}_{n=1}F_{n}\left(x\right)\right\Vert \le\left(1-C\right){}^{N/2+o(N)}\left\Vert x\right\Vert \;\text{a.s.}
\]
\end{cor}

\begin{proof}
Set $\alpha=1/2$ in \eqref{eq:d-1}.
\end{proof}
\begin{cor}[Random projections]
 Suppose $T\left(\omega\right)=P_{V\left(\omega\right)}$ is the
orthogonal projection onto a random closed subspace $V\left(\omega\right)\subset H$,
and let $G=\mathbb{E}[P_{V\left(\omega\right)}]$. If $\lambda_{\min}(G)=C>0$,
then $\{F_{n}\}$ forms an RNFF with frame bounds $A=C$ and $B$
given by \prettyref{thm:RNFF}. In this case the synthesis process
provides exact reconstruction of $x$ with exponential error decay.
\end{cor}

\begin{cor}[Randomized Kaczmarz]
 Let $a$ be a random vector in $\mathbb{R}^{d}$ and set 
\[
T\left(x\right)=x-\frac{\langle a,x\rangle}{\left\Vert a\right\Vert ^{2}}a.
\]
Then (A1)-(A3) hold with $C=1-\lambda_{\max}\left(\Sigma\right)$
where 
\[
\Sigma=\mathbb{E}\left[\frac{a\otimes a}{\left\Vert a\right\Vert ^{2}}\right].
\]
Hence $\{F_{n}\}$ forms an RNFF with exponential truncation rate
$\gamma=\tfrac{1}{2}\log\frac{1}{1-C}$. If $a/\|a\|$ is isotropic
in $\mathbb{R}^{d}$, then $C=1-\tfrac{1}{d}$ and 
\[
\left\Vert x-\sum\nolimits^{N}_{n=1}F_{n}\left(x\right)\right\Vert \lesssim\left(\frac{1}{\sqrt{d}}\right)^{N}\left\Vert x\right\Vert \quad a.s.
\]
\end{cor}

\begin{proof}
See \prettyref{exa:4-2}. 
\end{proof}
\begin{rem}
The construction above shows that random $\alpha$-averaged iterations
generate a nonlinear frame-like decomposition with strong quantitative
guarantees. Unlike classical frames, the atoms $F_{n}(x)$ depend
nonlinearly on the vector $x$ through the residual, yet the synthesis
is exact and stable. The single parameter $C$ governs both the frame
lower bound and the exponential sampling complexity. Specializations
recover randomized fusion frames and Kaczmarz methods, but the framework
extends to general firmly nonexpansive or averaged maps. 
\end{rem}

\bibliographystyle{amsalpha}
\bibliography{ref}

\end{document}